\documentclass[12pt]{amsart}
\makeatletter
\@namedef{subjclassname@2020}{\textup{2020} Mathematics Subject Classification}
\makeatother

\usepackage[margin=1in]{geometry}
\usepackage[english]{babel}
\usepackage[utf8]{inputenc}
\usepackage{subcaption}
\usepackage{amsmath}
\usepackage{amssymb}
\usepackage{amsfonts}
\usepackage{amsthm}
\usepackage{mathdots}
\usepackage{mathrsfs}
\usepackage[all]{xy}
\usepackage[pdftex]{graphicx}
\usepackage{color}
\usepackage{cite}
\usepackage{url}
\usepackage{indent first}
\usepackage[labelfont=bf,labelsep=period,justification=raggedright]{caption}
\usepackage[english]{babel}
\usepackage[utf8]{inputenc}
\usepackage{hyperref}
\usepackage[colorinlistoftodos]{todonotes}
\usepackage{tkz-fct}
\usepackage{tikz}
\usetikzlibrary{calc}
\usepackage{pgfplots}
\usepackage{multicol}
\PassOptionsToPackage{dvipsnames,svgnames}{xcolor}
\usepackage{textcomp}
\usepackage{bbm}
\usepackage[ruled,vlined]{algorithm2e}
\usepackage[shortlabels]{enumitem}

\newcommand{\NN}{\mathbb{N}}

\newcommand{\RR}{\mathbb{R}}

\newcommand{\mcG}{\mathcal{G}}

\newcommand{\mcO}{\mathcal{O}}

\renewcommand{\k}[1]{\ensuremath{\left({#1}\right)}}
\newcommand{\f}{\frac}
\newcommand{\ff}{\frac}
\newcommand{\mg}{{+\infty}}
\newcommand{\z}{\zeta}
\newcommand{\p}{\pi}
\newcommand{\twon}{\{2\}^N}
\newcommand{\nwon}{\{2\}^n}

\renewcommand{\a}{\alpha}

\theoremstyle{plain}
\newtheorem{thm}{Theorem}
\newtheorem{lemma}[thm]{Lemma}
\newtheorem{lem}[thm]{Lemma}
\newtheorem{cor}[thm]{Corollary}
\newtheorem{prop}[thm]{Proposition}

\theoremstyle{definition}
\newtheorem{defn}[thm]{Definition}

\newtheorem{conj}[thm]{Conjecture}

\newtheorem{ex}[thm]{Example}

\theoremstyle{remark}
\newtheorem{rem}[thm]{Remark}
\newtheorem{rmk}[thm]{Remark}

\numberwithin{equation}{section}
\numberwithin{thm}{section}

\begin{document}
\title[Evaluation of Odd Euler Sums, $\zeta(3,2,\ldots,2)$ and $t\left(3,2,\ldots,2\right)$]{Integral Evaluation of Odd Euler Sums, Multiple $t$-value $t\left(3,2,\ldots,2\right)$ and Multiple zeta value $\zeta(3,2,\ldots,2)$}

\author{Sarth Chavan}
\author{Masato Kobayashi}
\author{Jorge Layja}

\address{Euler Circle, Palo Alto, California 94306. sarth5002@outlook.com}
\address{Department of Engineering\\
Kanagawa University, 3-27-1 Rokkaku-bashi, Yokohama 221-8686, Japan. masato210@gmail.com}
\address{Department of Mathematics, Unversidad De Las Am\`ericas Puebla, Cholula, Mexico.}

	\maketitle
\begin{abstract}
    We construct an analytic approach to evaluate odd Euler sums, multiple zeta value $\zeta(3,2,\ldots,2)$ and multiple $t$-value $t\left(3,2,\ldots,2\right)$.  Moreover, we also conjecture a closed expression for multiple $t$-value $t\left(2,\ldots,2,1\right)$.
\end{abstract}

\section{Introduction and Main Results}
\subsection{Euler sums}

The classical \emph{Euler sums} are infinite series
\begin{equation}H\left(p_1, p_2, \ldots, p_k, q\right)=\sum_{n=1}^\mg\f{H_n^{\left(p_1\right)}H_n^{\left(p_2\right)}
\ldots H_n^{(p_k)}
}{n^q}\end{equation}
where $p_1,p_2,\ldots, p_k \in \NN$, $q \geqslant 2$  and 
$H_n^{\left(p\right)}=\sum_{k=1}^nk^{-p}$ is the generalized harmonic number. 

This theory indeed dates back to Euler and Goldbach, while recent research on this topic has been quite active. For an early
introduction and study on the evaluations of  classical Euler sums, the readers may consult in Flajolet and Salvy’s paper \cite{flsa}, in which they have developed a contour integral representation approach to the evaluation of Euler sums. 
For some recent progress, the readers are referred to \cite{Wang},\cite{Xu1}, \cite{Xu2}
and references therein. 

In particular, \emph{linear Euler sums} $H(p, q)$ with $p, q\geqslant2$  satisfy the \emph{reflection formula}
\[H(p, q)+H(q, p)=\z(p)\,\z(q)+\z(p+q)\]
as V\u{a}lean has discussed in \cite{valean}.
As a result of which, in the evaluation of $H(p, q)$, Riemann zeta values $\z(p)$ with $p\geqslant 2$ often show up with rational coefficients.

\subsection{Multiple zeta values and multiple $t$-values}
\emph{Multiple zeta values} (MZVs) are real numbers, originally defined by Euler, that
have been much studied in recent years because of their many surprising properties and the many places they appear in mathematics and mathematical
physics, ranging from periods of mixed Tate motives \cite{br} to values of Feynman
integrals in perturbative quantum field theory. 


In fact, multiple zeta values (MZVs) are simply a variant of the Euler sums. For positive integers $i_1, i_2 \ldots, i_k$ where $i_1\geqslant 2$, the infinite series 
\begin{equation}
\label{mzv}
\z\left(i_1, i_2, \ldots, i_k\right)
=
\sum_{
\substack{
n_1>n_2>\cdots >n_k\geqslant 1
}
}
\frac{1}{n_1^{i_1}n_2^{i_2}\cdots n_k^{i_k}}
\end{equation}
is usually called a \emph{multiple zeta value}. 
Weight of the index $\left(i_1, i_2, \ldots, i_k\right)$ 
is $i_1+i_2+\cdots+i_k$ and its depth is $k$.
The study of multiple zeta values began in the early 1990s with the works of Hoffman \cite{k1}
and Zagier \cite{k2}, and has attracted a lot of research in the last two decades. For detailed history
and applications, the readers may consult  the book of Zhao \cite{Zhao}.

One of the main problems in this topic is to understand the
$\mathbf{Q}$-linear dependence of MZVs.
Some MZVs are a combination of single zeta values 
while according to Brown, there certainly seems to exist exceptions such as $\z(5, 3)$ \cite{br}. We usually call the MZVs $\z(i_1, i_2, \ldots, i_k)$
with $i_j\in \{2, 3\}$ \emph{Hoffman elements}. In particular, F. Brown \cite{br} proved that every multiple zeta value is a $\mathbf{Q}$-linear combination of Hoffman elements. Moreover,  Zagier \cite{za} showed explicit relations among Hoffman elements in the form $\z(2, \ldots, 2, 3, 2, \ldots, 2)$.

In a recent paper \cite{hoffmann}, Hoffman introduced and studied an \emph{odd} variant of the MZVs:
\begin{equation}
t\left(i_1, i_2, \ldots, i_k\right)
=
\sum_{
	\substack{
		n_1>n_2>\cdots >n_k\geqslant 1\\
		n_j\, \mathrm{odd}
	}
}
\frac{1}{n_1^{i_1}n_2^{i_2}\ldots n_k^{i_k}} = \sum_{n_1>n_2>\cdots >n_k\geqslant 1} \dfrac{1}{\prod_{j=1}^{k}\left(2n_j - 1\right)^{i_j}}
\end{equation}
which are called the \emph{multiple t-values} (MtVs). As showed in \cite[Corollaries 4.1 and 4.2]{hoffmann}, the MtVs are
reducible to a 	linear combinations of alternating (or colored) MZVs.	

\vspace{0.1in}
There are indeed many conjectures concerning the arithmetic nature of these numbers. 
Let $(d_n)_{n\geqslant1}$ be the \emph{Padovan sequence} defined by 
$d_1=0, 
d_2=1, d_3=1$ and $d_n=d_{n-2}+d_{n-3}$ for $n\geqslant4$ and let $(f_n)_{n\geqslant1}$ be the Fibonacci sequence with $f_1=f_2=1$, $f_n=f_{n-1}+f_{n-2}$ for $n\geqslant 3$. Let $\mathcal{Z}_n$ and $\mathcal{T}_n$ denote the $\mathbf{Q}$-span of all MZVs  and MtVs of weight $n$ respectively. Then, we have the following conjectures by M. Hoffmann \cite{hoffmann}, B. Saha \cite{saha} and D. Zagier \cite{za}.

\begin{conj}\hfill
\begin{enumerate}
    \item {(Zagier \cite{za})} $\dim \mathcal{Z}_n=d_n$.
    \item {(Hoffman \cite{hoffmann})} $\dim \mathcal{T}_n=f_n$.
    \item {(Hoffman \cite{hoffmann})} 
    The following set of weight $n$ MZVs forms a $\mathbf{Q}$-basis of $\mathcal{Z}_n$.
\[
\{\z(i_1, i_2, \ldots, i_k)\mid i_1+i_2+\cdots +i_k=n, i_j\in \{2, 3\}
\}.
\] 
\item {(Saha \cite{saha})}
The following set of weight $n$ MtVs forms a $\mathbf{Q}$-basis of $\mathcal{T}_n$ for $n\geqslant2$.
\[
\{
t\left(a_1+1, a_2, \dots, a_r\right)\mid a_1+a_2+\cdots +a_r=n-1, a_i\in\{1, 2\}, 1\leqslant r\leqslant n-1
\}.
\]

\end{enumerate}
\end{conj}


We often  discuss MZVs and MtVs independently since the algebra of multiple $t$-values is quite different
in some ways from the algebra of multiple zeta values. Both the duality
theorem and the \emph{double shuffle relations} \cite{shuffle} of MZVs are
missing for MtVs.
However, there should be rich connections between these two theories. 
In this article, we show that 
we can compute 
$\z(3, 2, \ldots, 2)$ and 
$t\left(3, 2, \ldots, 2\right)$ using the same method.

\subsection{Central binomial sums}
\emph{Central binomial sums} play an important role in number theory.
Informally speaking, it is  an infinite series involving  $\binom{2n}{n}$.
Lehmer \cite{le} discussed the following two types of such sums:
\[\textrm{I}.\, 
\sum_{n=0}^{+\infty} a_n \binom{2n}{n}, \quad 
\textrm{II}.\,
\sum_{n=0}^{+\infty} a_n \binom{2n}{n}^{-1}.
\]
Moreover, he also presented connections between such series and 
Maclaurin series of $\arcsin\left(z\right)$ and $\arcsin^2\left(z\right)$.
Some other examples include
\[
	\sum_{n=1}^{\mg}\ff{1}{n^2\binom{2n}{n}}=\ff{1}{3}\,\z(2), \quad
	\sum_{n=1}^{\mg}\ff{(-1)^{n-1}}{n^3\binom{2n}{n}}=\ff{2}{5}\,\z(3),\quad
	\sum_{n=1}^{\mg}\ff{1}{n^4\binom{2n}{n}}=\ff{17}{36}\,\z(4)
\]
as they arise in the work of Ap\`{e}ry \cite{apery} and van der Poorten \cite{van} to prove irrationality of $\z(3)$. 
Along this line, it is natural to consider an application of such sums for MZVs and MtVs.
\subsection{Main results}
Our main results are at the crossroads of these three topics. Moreover,  our idea throughout is more analytic; we seek Euler sums, MZVs and MtVs by \emph{integral evaluation}. This is indeed a powerful method as seen in Doelder \cite{de} and V\u{a}lean \cite{valean}.

In Section \ref{sec2} and Section \ref{sec3}, $t\left(3,2,\ldots,2\right)$ and $\z(3,2,\ldots,2)$ are evaluated using the inverse sine integral (\ref{mainint}) and the nested sum representations for integer powers of the inverse sine function \cite[Equations (1) and (2)]{boch}.
%
In particular, it is shown that
$\z(3, 2, \ldots ,2)$ and 
$t\left(3, 2, \ldots ,2\right)$ are rational polynomials of 
$\pi$ and odd zeta values $\z(2k+1)$ with $k\geqslant1$.


In Section \ref{sec4}, \emph{multiple mixed value} $\mu\left(i_k, i_{k-1}, \ldots,i_1\right)$ is defined and the Maclaurin series expansion for integer powers of inverse hyperbolic tangent function 
\cite[(6.3)]{glq} is used to evaluate $\mu\left(2,1,\ldots,1\right)$.
Finally, in Section \ref{sec5}, 
 \emph{odd Euler sums} $\mcO\left(p,q\right)$ and $\mathcal{B}\left(p,q\right)$ are defined and a reflection formula for $\mcO\left(p,q\right)$ and $\mathcal{B}\left(p,q\right)$ is established to evaluate $\mcO\left(q,q\right)$ and $\mathcal{B}\left(q,q\right)$. Later, an analytic approach is implemented to compute $\mcO\left(p,q\right)$ and $\mathcal{B}\left(p,q\right)$ for small values of $p$ and $q$. Moreover, integral reduction of $\mcO\left(p,q\right)$ and $\mathcal{B}\left(p,q\right)$ are also obtained.

Beyond these specific cases, we hope that the tools used in this study will be useful to the
readers in their endeavour of evaluation of more interesting series and integrals. 

 



\subsection{Notation}

Let 
\begin{align*}
	\left(2n\right)!!=\prod_{k=0}^{n-1}\left(2n-2k\right)
\,\,\,	\text{and}\,\,\,
	\left(2n-1\right)!!=\prod_{k=0}^{n-1}\left(2n-2k-1\right).
\end{align*}
Moreover, for notational convenience, let
\[\z(3,\underbrace{2,2,\ldots,2}_{n \,\text{times}}) = \z\left(3,\nwon\right) \,\,\,	\text{and}\,\,\, t\,(3,\underbrace{2,2,\ldots,2}_{n\,\text{times}}) = t\left(3,\nwon\right).\]


\section{Evaluation of Multiple $t$--value $t\left(3,2,\ldots,2\right)$} \label{sec2}

Let us first begin with several Definitions and Propositions.


\begin{defn}\label{W}
Let $\mathbf{R}[[z]]$ denote the set of 
convergent power series with real coefficients. 
For $f(z)\in \mathbf{R}[[z]]$, define $\mathbb{W}:\mathbf{R}[[z]]\to \mathbf{R}[[z]]$ to be 
\[
\mathbb{W}f(z)=
\int_0^1 \frac{f(xz)}{\sqrt{1-x^2}}\,\mathrm{d}x.
\]
\end{defn}
It is now helpful to understand the function $\mathbb{W}$ coefficientwise.

\begin{prop}
\label{l1}
Let $f(z)\in \mathbf{R}[[z]]$. 
Moreover, suppose it can be expressed in the form 
\[f(z)= \sum_{k=0}^{\mg}\f{\left(2k-1\right)!!}{\left(2k\right)!!}\,\mathcal{C}_{2k+1}\,z^{2k+1}.
\]
Then, the following identity holds. 
\[\mathbb{W}f(z)=\sum_{k=0}^{\mg}\f{\mathcal{C}_{2k+1}}{2k+1}\,z^{2k+1}.\]
\end{prop}
\begin{proof}
From Definition \ref{W}, it follows that
\[
\mathbb{W}f(z)=
\sum_{k=0}^{+\infty}
\f{\left(2k-1\right)!!}{\left(2k\right)!!}
\,\mathcal{C}_{2k+1}\,z^{2k+1}
{\int_0^1
\f{x^{2n+1}}{\sqrt{1-x^2}}\,\mathrm{d}x}
\]
\[
=\sum_{k=0}^{+\infty}
\f{\left(2k-1\right)!!}{\left(2k\right)!!}\f{\left(2k\right)!!}{\left(2k-1\right)!!}\,\mathcal{C}_{2k+1}\,z^{2k+1}=
\sum_{k=0}^{\mg}
\f{\mathcal{C}_{2k+1}}{2k+1}\,z^{2k+1}\]
which is the desired result. 
\end{proof}

\begin{prop}\label{wallis}
Let $f(z)\in \mathbf{R}[[z]]$ and $\alpha \in \mathbf{R}$. 
If $f(0)=0$, then we have
\[
\mathbb{W}\k{
\int_{0}^{\alpha}{\f{f(z)}{z}}\,\mathrm{d}z}=
\int_{0}^{1}
	{\f{f(\alpha x)\arccos{\left(x\right)}}{x}}\,\mathrm{d}x.
	\]
\end{prop}

\begin{proof}
Notice that, if $\alpha=0$, then both sides are 0. 
Otherwise, we have
\begin{align*}
	\mathbb{W}\k{
\int_{0}^{\alpha}{\f{f(z)}{z}}\,\mathrm{d}z}
&=
\int_{0}^{1}{
\int_{0}^{\alpha u}{\f{f(z)}{z}}
}\f{1}{\sqrt{1-u^{2}}}\,\mathrm{d}z\,\mathrm{d}u
\end{align*}
\[=
\int_{0}^{\alpha}{
{
\f{f(z)}{z}
}
\int_{z/\alpha}^{1}
\f{1}{\sqrt{1-u^{2}}}
}\,\mathrm{d}u\,\mathrm{d}z=\int_{0}^{\alpha}
	{\f{f(z)}{z}}\arccos\left({\f{z}{\alpha }}\right)\mathrm{d}z
	=\int_{0}^{1}
	{\f{f(\alpha x)\arccos{\left(x\right)}}{x}}\,\mathrm{d}x\]
as desired.
\end{proof}

\begin{prop}[J.M. Borwein--Chamberland{\cite[(1.3), (1.4)]{boch}}]
\label{asin}
For $|x|\leqslant 2$ and $N\geqslant0$
\[
\f{1}{\left(2N+1\right)!}\left[\arcsin\left(\dfrac{x}{2}\right)\right]^{2N+1}=
\sum_{k=0}^{\mg}\f{G_N(k)}{2^{2k}\left(2k+1\right)}
\binom{2k}{k}\k{\f{x}{2}}^{2k+1}
\]
where $G_0(k)=1$ and 
\[
G_N(k)=
\sum_{n_1=0}^{k-1}\f{1}{\left(2n_1+1\right)^2}
\sum_{n_2=0}^{n_1-1}\f{1}{\left(2n_2+1\right)^2}
\,\,\,\cdots
\sum_{n_N=0}^{n_{N-1}-1}\f{1}{\left(2n_N+1\right)^2}.\]
\end{prop}
\begin{rmk}
Notice that the boundary $|x|=\pm2$ is inclusive here.

\end{rmk}

\begin{prop}\label{as_int}
Let 
\begin{equation}\label{mainint}I(N) = \int_0^1 \dfrac{\arcsin^N(z)}{z}\,\mathrm{d}z \quad\left(N \in \NN\right)\end{equation}
Then, the following identities hold.
\begin{equation}
\label{I2N+1}
I(2N+1) = \dfrac{\left(2N+1\right)!}{2^{2N+1}}\sum_{j=0}^{N}\dfrac{\left(-1\right)^j\pi^{2N+1-2j}}{\left(2N+1-2j\right)!}\,\eta\left(2j+1\right) 
\end{equation}
\begin{equation}
\label{I2N}
I(2N) = \dfrac{\left(2N\right)!}{2^{2N}}
\left(\sum_{j=0}^{N-1}\dfrac{\left(-1\right)^j\pi^{2N-2j}}{\left(2N-2j\right)!}\,\eta\left(2j+1\right) + \left(-1\right)^N2\left(1-2^{-2N-1}\right)\zeta(2N+1)\right)
\end{equation}
where  $\eta\left(m\right) = \left(1-2^{1-m}\right)\zeta(m)$ represents the Dirichlet eta function with $\eta\left(1\right)=\log(2)$. 
\end{prop}
\begin{proof}
Buhler and Crandall in their very recent article\cite[page 280]{bucr} proved that
\[
J(n) =
\displaystyle\int_{0}^{1/2}{z^{n}\cot\left(\pi z\right)}\,\mathrm{d}z
\]
\[=
\ff{n!}{2^{n}}
\k{
\sum_{1\leqslant k\leqslant n, \, k \,\textrm{odd}}
\ff{\left(-1\right)^{(k-1)/2}\eta\left(k\right)}{\pi^{k}\left(n-k+1\right)!}
}
+
\cos\left(\frac{n\pi}{2}\right)
\ff{\left(4n!\right)\left(1-2^{-n-1}\right)}{\left(2\p\right)^{n+1}}\,\z(n+1).
\]
Note that we corrected the sign $\frac{1+\left(-1\right)^{n}}{2}$ to $\cos\left(\frac{n\pi}{2}\right)$. Substituting $x=\sin\left(\pi z\right)$ produces
\[
J(n)=\int_{0}^{1/2}{z^{n}\cot\left(\pi z\right)}\,\mathrm{d}z=\ff{1}{\pi^{n+1}}
\displaystyle\int_{0}^{1}{\ff{\arcsin^{n}(x)}{x}}\,\mathrm{d}z
=\ff{I(n)}{\p^{n+1}}.
\]
Writing down the cases for $n=2N+1$ and $n=2N$ with $k=2j+1$ yields
\[I(2N+1) = \dfrac{\left(2N+1\right)!}{2^{2N+1}}\sum_{j=0}^{N}\dfrac{\left(-1\right)^j\pi^{2N+1-2j}}{\left(2N+1-2j\right)!}\,\eta\left(2j+1\right),\]
\[I(2N) = \dfrac{\left(2N\right)!}{2^{2N}}
\left(\sum_{j=0}^{N-1}\dfrac{\left(-1\right)^j\pi^{2N-2j}}{\left(2N-2j\right)!}\,\eta\left(2j+1\right) + \left(-1\right)^N2\left(1-2^{-2N-1}\right)\zeta(2N+1)\right)\]
which is the desired result. 
\end{proof}
\begin{cor}
\label{arcsin_ints}
First few values of $I(N)$ are as follows.
\begin{equation}
     I(1)= \dfrac{\pi}{2}\,\log\left(2\right),
\end{equation}
\begin{equation}
\label{I2}
     I(2)= \dfrac{\pi^2}{4}\,\log\left(2\right)-\dfrac{7}{8}\,\zeta(3), 
    \end{equation}
\begin{equation}
     I(3)= \dfrac{\pi^3}{8}\,\log\left(2\right)-\dfrac{9\pi}{16}\,\zeta(3),
     \end{equation}
     \begin{equation}
     \label{I4}
     I(4)= \dfrac{\pi^4}{16}\,\log\left(2\right)-\dfrac{9\pi^2}{16}\,\zeta(3) + \dfrac{93}{32}\,\zeta(5),\end{equation}
     \begin{equation}
     I(5)= \dfrac{\pi^5}{32}\,\log\left(2\right) - \dfrac{15\pi^3}{32}\,\zeta(3) + \dfrac{225\pi}{64}\,\zeta(5),  \end{equation}
     \begin{equation}
     \label{I6}
I(6)=\dfrac{\pi^6}{64}\,\log\left(2\right) - \dfrac{45\pi^4}{128}\,\zeta(3) + \dfrac{675\pi^2}{128}\,\zeta(5) -\dfrac{5715}{256}\,\zeta(7),\end{equation}
     \begin{equation}
     I(7)=\dfrac{\pi^7}{128}\log\left(2\right)-\dfrac{63\pi^5}{256}\,\z(3)+\dfrac{1575\pi^3}{256}\,\z(5)-\dfrac{19845\pi}{512}\,\z(7),
\end{equation}
\begin{equation}
I(8)=\dfrac{\pi^{8}}{256}\log\left(2\right)  - \dfrac{21\pi^6}{128}\,\zeta(3) + \dfrac{1575\pi^4}{256}\,\zeta(5) - \dfrac{19845\pi^2}{256}\,\zeta(7) + \dfrac{160965}{512}\,\zeta(9).
\end{equation}
\end{cor}

\begin{thm}\label{t322}
For $N\geqslant1$, we have
\[
t\left(3, \twon\right)
\]
\[
=\ff{1}{2^{2N+2}}
\left[
\sum_{j=1}^{N}\ff{\left(-1\right)^{j+1}\left(2j\right)\p^{2N+2-2j}}{\left(2N+2-2j\right)!}\,\eta(2j+1)
+\left(-1\right)^N2\left(2N+2\right)\left(1-2^{-2N-3}\right)\z(2N+3)\right].\]
\end{thm}

\begin{proof}
From Proposition \ref{asin}, it follows that
\[
\ff{\arcsin^{2N+1}(x/2)}{(2N+1)!}
=
\sum_{k>n_{1}>\cdots >n_{N}\geqslant0}
\ff{1}{\left(2k+1\right)\left(2n_{1}+1\right)^{2}\ldots \left(2n_{N}+1\right)^{2}}
\ff{\left(2k-1\right)!!}{\left(2k\right)!!}
\k{\ff{x}{2}}^{2k+1}
\]
where we have simply used the identity
\[\dfrac{1}{2^{2k}}\dbinom{2k}{k}=\f{\left(2k-1\right)!!}{\left(2k\right)!!}\]
For notational convenience, let $z=x/2$. Therefore we have
\[
\ff{\arcsin^{2N+1}(z)}{\left(2N+1\right)!}
=
\sum_{k>n_{1}>\cdots >n_{N}\geqslant0}
\ff{1}{\left(2k+1\right)\left(2n_{1}+1\right)^{2}\ldots \left(2n_{N}+1\right)^{2}}
\ff{\left(2k-1\right)!!}{\left(2k\right)!!}\,
z^{2k+1}.
\]
Dividing both sides by $z$ and integrating termwise from $0$ to $\alpha$ with respect to $z$ produces
\[
\displaystyle\int_{0}^{\a}{
\ff{\arcsin^{2N+1}(z)}{z\left(2N+1\right)!}
}\,\mathrm{d}z
=
\sum_{k>n_{1}>\cdots >n_{N}\geqslant0}
\ff{1}{\left(2k+1\right)^{2}\left(2n_{1}+1\right)^{2}\ldots
\left(2n_{N}+1\right)^{2}}
\ff{\left(2k-1\right)!!}{\left(2k\right)!!}\,
\a^{2k+1}.
\]
Next, we use Proposition \ref{l1} to deduce that
\[
\mathbb{W}
\k{\displaystyle\int_{0}^{\a}{
\ff{\arcsin^{2N+1}(z)}{z\left(2N+1\right)!}
}\,\mathrm{d}z}
=
\sum_{k>n_{1}>\cdots >n_{N}\geqslant0}
\ff{1}{\left(2k+1\right)^{3}\left(2n_{1}+1\right)^{2}\ldots \left(2n_{N}+1\right)^{2}}\,\alpha^{2k}
\]
Substituting $\alpha=1$ and combining Proposition \ref{wallis} and \ref{as_int}, we finally conclude that
\[
t\left(3, \twon\right)=
\ff{1}{(2N+1)!}
\displaystyle\int_{0}^{1}{
\ff{\arcsin^{2N+1}(z)\arccos\left(z\right)}{z}
}\,\mathrm{d}z=
\dfrac{\pi}{2}\left(\ff{I(2N+1)}{\left(2N+1\right)!}\right)
-\dfrac{I(2N+2)}{\left(2N+1\right)!}
\]
\[
=\ff{1}{2^{2N+2}}
\left[
\sum_{j=1}^{N}\ff{\left(-1\right)^{j+1}\left(2j\right)\p^{2N+2-2j}}{\left(2N+2-2j\right)!}\,\eta(2j+1)
+\left(-1\right)^N2\left(2N+2\right)\left(1-2^{-2N-3}\right)\z(2N+3)\right]
\]
which is the desired result. This completes the proof of Theorem \ref{t322}. 
\end{proof}


\begin{cor}
Substituting $N=2,3$ in Theorem \ref{t322} produces
\[
t\left(3, 2, 2\right)=
\frac{\pi}{2}\left(\dfrac{I(5)}{120}\right)-\dfrac{I(6)}{120}
=\dfrac{\pi^4}{1024}\,\zeta(3) - \dfrac{15\pi^2}{512}\,\zeta(5) + \dfrac{381}{2048}\,\zeta(7)
=0.002109185\ldots
\]
\[
t\left(3, 2, 2, 2\right)
=\dfrac{\pi^6}{122880}\,\zeta(3) - \dfrac{5\pi^4}{8192}\,\z(5) + \dfrac{189\pi^2}{16384}\,\z(7) + \dfrac{511}{8192}\,\z(9)
=0.00005499616\ldots
\]
\end{cor}

\section{Evaluation of Multiple Zeta Value $\zeta(3,2,\ldots,2)$} \label{sec3}

It turns out that we can use the same method as in the previous section to evaluate  multiple zeta value $\z(3, 2, \ldots, 2)$. Although this is merely a special case of Zagier's work \cite[p.981, Theorem 1]{za},  we present our proof since our approach is quite different. 

\begin{prop}[Borwein--Chamberland{\cite[(1.1), (1.2)]{boch}}]
\label{l2}
For $|x|\leqslant 2$ and $N\geqslant1$, we have 
\[
\ff{1}{\left(2N\right)!}\left[\arcsin{\left(\dfrac{x}{2}\right)}\right]^{2N}=
\sum_{k=1}^{\mg}\ff{H_{N}(k)}{\binom{2k}{k}}\ff{x^{2k}}{k^{2}}
\]
where $H_{1}(k)=1/4$ and 
\[
H_{N+1}(k)=
\ff{1}{4}
\sum_{n_{1}=1}^{k-1}\ff{1}{\left(2n_{1}\right)^{2}}
\sum_{n_{2}=1}^{n_{1}-1}\ff{1}{\left(2n_{2}\right)^{2}}
\,\,\,\cdots
\sum_{n_{N}=1}^{n_{N-1}-1}\ff{1}{\left(2n_{N}\right)^{2}}.
\]
\end{prop}
\begin{prop}
\label{even}
Let $f(z) \in \mathbf{R}[[z]]$. Moreover, suppose it can be expressed in the form
\[f(z)= \sum_{k=0}^{\mg}
{\f{\left(2k\right)!!}{\left(2k-1\right)!!}}
\,\mathcal{C}_{2k}\,z^{2k}.
\]
Then, the following identity holds.
\[\mathbb{W}f(z)=
\f{\pi}{2}
\sum_{k=0}^{\mg}\mathcal{C}_{2k}\,z^{2k}.\]
\end{prop}
\begin{proof}
The proof is quite similar to that of Proposition \ref{l1}. 
Notice that 
\[
\int_0^1 \ff{z^{2k}}{\sqrt{1-z^2}}\,\mathrm{d}z=\f{\pi}{2}\f{\left(2k-1\right)!!}{\left(2k\right)!!}
\]
which explains the appearance of $\pi/2$ in $\mathbb{W}f(z)$. 
\end{proof}
\begin{thm}\label{z322}
For $N \geqslant 0$, we have
\[\z\left(3, \twon\right)\]
\[
=2\left[\sum_{j=1}^{N}\frac{\left(-1\right)^{j+1}\left(2j\right)\pi^{2N+2-2j}}{\left(2N+3-2j\right)!}\,\eta\left(2j+1\right)
-\left(-1\right)^{N}
\left[1-\left(1-2^{-2N-2}\right)\left(2N+2\right)\right]\z(2N+3)\right].
\]
\end{thm}

\begin{proof}
From Proposition \ref{l2}, it follows that
\[
\ff{1}{\left(2N+2\right)!}\left[\arcsin\left(\dfrac{x}{2}\right)\right]^{2N+2}
=
\ff{1}{2^{2N+2}}
\sum_{k>n_{1}>\cdots >n_{N}}
\ff{1}{k^{2}\,n_{1}^{2}\,n_2^2\ldots n_{N}^{2}}
\ff{\left(2k\right)!!}{\left(2k-1\right)!!}
\k{\ff{x}{2}}^{2k}.
\]
For notational convenience, let $z=x/2$, therefore we have
\[
\ff{\arcsin^{2N+2}(z)}{\left(2N+2\right)!}
=
\ff{1}{2^{2N+2}}
\sum_{k>n_{1}>\cdots >n_{N}}
\ff{1}{k^{2}\,n_{1}^{2}\,n_2^2\ldots n_{N}^{2}}
\ff{\left(2k\right)!!}{\left(2k-1\right)!!}\,
z^{2k}.
\]
Next, we divide both sides by $z$ and integrate them termwise from 0 to $\alpha$ to get
\[
\displaystyle\int_{0}^{\a}{
\ff{\arcsin^{2N+2}(z)}{z\left(2N+2\right)!}
}\,\mathrm{d}z
=
\ff{1}{2^{2N+3}}
\sum_{k>n_{1}>\cdots >n_{N}}
\ff{1}{k^{3}\,n_{1}^{2}\,n_2^2\ldots n_{N}^{2}}
\ff{\left(2k\right)!!}{\left(2k-1\right)!!}
\,\a^{2k}.
\]
Next, we use Proposition \ref{even} to deduce that
\[
\mathbb{W}
\k{\displaystyle\int_{0}^{\a}{
\ff{\arcsin^{2N+2}(z)}{z\left(2N+2\right)!}
}\,\mathrm{d}z}
=
\ff{\pi}{2^{2N+4}}
\sum_{k>n_{1}>\cdots >n_{N}}
\ff{1}{k^{3}\,n_{1}^{2}\,n_2^2\ldots n_{N}^{2}}
\,\a^{2k}.
\]
Substituting $\a=1$ and combining Proposition \ref{wallis}, Proposition \ref{as_int} finally produces
\[
\zeta\left(3,\twon\right)=\ff{2^{2N+4}}{\left(2N+2\right)!}\left[\dfrac{I(2N+2)}{2} - \dfrac{I(2N+3)}{\pi}\right]\]
\[
=2\left[\sum_{j=1}^{N}\frac{\left(-1\right)^{j+1}\left(2j\right)\pi^{2N+2-2j}}{\left(2N+3-2j\right)!}\,\eta\left(2j+1\right)
-\left(-1\right)^{N}
\left[1-\left(1-2^{-2N-2}\right)\left(2N+2\right)\right]\z(2N+3)\right]
\]
as desired. This completes the proof of Theorem \ref{z322}. 
\end{proof}

 \begin{cor}\label{spec}
Substituting $N = 1,2,3$ in Throem \ref{z322} produces
   \[\zeta(3,2) = \dfrac{\pi^2}{2}\,\zeta(3) - \dfrac{11}{2}\,\zeta(5)=
  0.22881039\ldots,
   \]
   \[\zeta(3,2,2) = \dfrac{\pi^4}{40}\,\zeta(3)-\dfrac{5\pi^2}{4}\,\zeta(5) + \dfrac{157}{16}\,\zeta(7)
   =0.02912562\ldots,
   \]
   \[\zeta(3,2,2,2) = \dfrac{\pi^6}{1680}\,\zeta(3)-\dfrac{\pi^4}{16}\,\zeta(5)+\dfrac{63\pi^2}{32}\,\zeta(7) -\dfrac{223}{16}\,\zeta(9)
  =0.00252145\ldots.
   \]
   \\
We have also verified these identities numerically at the computational website EZ-Face \cite{ez}.
\end{cor}

\begin{cor} Moreover
$$\z\left(3, \twon\right)\in 
\mathbf{Q}\left[\pi, \z(3), \z(5), \dots, \z(2N+3)\right]_{2N+3}
$$
where $\mathbf{Q}\left[\pi, \z(3), \z(5), \ldots, \z(2N+3)\right]_{2N+3}$
 is the set of all elements of degree $2N+3$ in the rational polynomial ring in $\pi, \z(3), \z(5), \z(7),\ldots, \z(2N+3)$ with grading $\deg\left(\p\right)=1$ and $\deg\left(\z(2j+1)\right)=2j+1$. Indeed, the same is true for multiple $t$-value 
$t\left(3, \twon\right)$.
\end{cor}
\section{Evaluation of Multiple Mixed Value $\mu\left(2,1,\ldots,1\right)$} \label{sec4}

\begin{defn}
For positive integers $i_1, i_2 \ldots, i_k$, define 
\[
\mu\left(i_k, i_{k-1}, \ldots,i_1\right)
=\sum_{
\substack{n_k>n_{k-1}>\cdots >n_1\\
n_j \,\,\equiv \,\,j \,\,\textrm{(mod $2$)}
}}
\dfrac{1}{n_k^{i_k}n_{k-1}^{i_{k-1}} \ldots n_1^{i_1}}.
\]
\end{defn}
Notice the little reversal of indices. If $i_k\geqslant2$, this infinite sum is certainly convergent 
since it is a partial sum of $\z(i_k, i_{k-1}, \ldots, i_1)$. 
Let us call this sum a \emph{multiple mixed value}.

It is not so immediate to evaluate such sums. 
However, quite recently, Guo--Lim--Qi in their preprint \cite{glq}, announced that they found Maclaurin series expansion for integer powers of inverse hyperbolic tangent function which is helpful in the evaluation $\mu\left(2, 1, \ldots,1\right)$.

\begin{prop}[{Guo--Lim--Qi \cite[(6.3)]{glq}}]
For $N\geqslant0$, $\ell_N=k$ and $|z|<1$, we have
\begin{equation}
\label{atanh}
\ff{\textnormal{arctanh}^{N}(z)}{N!}
=
\sum_{k=0}^{+\infty}
\left(\prod_{m=1}^{N-1}\sum_{\ell_m=0}^{\ell_{m+1}}\dfrac{1}{2\ell_m + m}\right)\dfrac{z^{2k+N}}{2k+N}
=
\sum_{
\substack{n_N>n_{N-1}>\cdots >n_1\\
n_j \,\,\equiv \,\,j \,\,\textnormal{(mod $2$)}
}}
\dfrac{z^{n_N}}{n_Nn_{N-1} \ldots n_1}.
\end{equation}
\end{prop}


\begin{thm}\label{e211}
For $N\geqslant1$, the following identity holds
\[\mu\left(2, \{1\}^{N-1}\right)=\dfrac{2^{N+1}-1}{2^{2N}}\,\zeta(N+1).\]
\end{thm}

\begin{proof}
Integrating equation (\ref{atanh}) with respect to $z$ from $0$ to $1$ produces
\[\mu\left(2,\{1\}^{N-1}\right) =\dfrac{1}{N!} \int_0^1\dfrac{\textrm{arctanh}^N(z)}{z}\,\mathrm{d}z.\]
Let us call this integral  $K(N)$.
We can indeed evaluate $K(N)$ by simply using the identity $$\mathrm{arctanh}\left(z\right) = -\dfrac{1}{2}\log\left(\dfrac{1-z}{1+z}\right).$$ 
We have
\[K(N)=\int_0^1 \dfrac{\mathrm{arctanh}^N(z)}{z}\,\mathrm{d}z = \dfrac{\left(-1\right)^N}{2^N}\int_0^1\dfrac{1}{z}\log^N\left(\dfrac{1-z}{1+z}\right)\mathrm{d}z=\dfrac{\left(-1\right)^N}{2^{N-1}}\int_0^1\dfrac{\log^N(z)}{1-z^2}\,\mathrm{d}z\]

\[=\dfrac{\left(-1\right)^N}{2^{N-1}}\sum_{k=0}^{+\infty}\int_0^1z^{2k}\log^{N}(z)\mathrm{d}z = \dfrac{\left(-1\right)^N}{2^{N-1}}\sum_{k=0}^{+\infty}\dfrac{\left(-1\right)^N N!}{\left(2k+1\right)^{N+1}}=\dfrac{N!\left[2^{N+1}-1\right]}{2^{2N}}\,\zeta(N+1).\]
Therefore, putting all things together produces the desired result. \end{proof}

\begin{ex}
	Substituting $N=1,2,3$ in Theorem \ref{e211} produces
\[
\mu\left(2, 1\right)=\ff{7}{16}\,\z(3), \,\,\, 
\mu\left(2, 1, 1\right)=\ff{15}{64}\,\z(4) = \dfrac{\pi^4}{384}, \,\,\, 
\mu\left(2, 1, 1, 1\right)=\ff{31}{256}\,\z(5).\]
\end{ex}
Indeed, $K(2N)$ appears to be a part (in fact, the boundary term) of $I(2N)$, that is
\begin{equation}
	\label{I2N}
	I(2N) = \dfrac{\left(2N\right)!}{2^{2N}}
	\left(\sum_{j=0}^{N-1}\dfrac{\left(-1\right)^j\pi^{2N-2j}}{\left(2N-2j\right)!}\,\eta\left(2j+1\right)\right) + \left(-1\right)^NK(2N).
\end{equation}
Notice that Hoffman \cite[Appendix]{hoffmann} highlights the relations
\[
t\left(2, 1\right)=t\left(2\right)\log(2)-\f{1}{2}\,t\left(3\right),\]
\[
t\left(2, 2, 1\right)=\f{1}{8}\,t\left(5\right)-\f{1}{14}\,t\left(2\right)t\left(3\right)+\f{1}{4}\,t\left(4\right)\log(2),\]
\[
t\left(2, 2, 2, 1\right)=
-\f{1}{32}\,t\left(7\right)-\f{3}{56}\,t\left(3\right)t\left(4\right)+\f{15}{248}\,t\left(2\right)t\left(5\right)+\f{1}{48}
\,t\left(6\right)\log(2).
\]
\\
With equations (\ref{I2}), (\ref{I4}), (\ref{I6}) and using $t\left(i\right)=\left(1-2^{-i}\right)\z(i)$, we observe that
\[
t\left(2, 1\right)=
\f{1}{2}
\k{
	\f{\pi^2}{4}\,\log(2)-\f{7}{8}\,\z(3)
}=
\f{I(2)}{2!},\]
\[
t\left(2, 2, 1\right)=
\f{1}{24}
\k{
	\f{\pi^4}{16}\,\log(2)-\f{9\pi^2}{16}\,\z(3)+
	\f{93}{32}\,\z(5)
}
=
\f{I(4)}{4!},\]
\[
t\left(2, 2, 2, 1\right)=
\ff{1}{720}
\k{\dfrac{\pi^6}{64}\,\log(2) - \dfrac{45\pi^4}{128}\,\zeta(3) + \dfrac{675\pi^2}{128}\,\zeta(5) -\dfrac{5715}{256}\,\zeta(7)}
=
\f{I(6)}{6!}.\]
This pattern seems to continue further. However, we have not figured out why yet. 
\begin{conj}
	Let $N \in \mathbb{N}$, then the following identity holds.
	\[t\left(\twon,1\right)=\dfrac{I(2N)}{\left(2N\right)!}=\dfrac{1}{2^{2N}}
	\left(\sum_{j=0}^{N-1}\dfrac{\left(-1\right)^j\pi^{2N-2j}}{\left(2N-2j\right)!}\,\eta\left(2j+1\right)+ \left(-1\right)^N2\left(1-2^{-2N-1}\right)\zeta(2N+1)\right).\]
	Notice that $t\left(\twon,1\right)$  is a MtV appearing in Saha's conjecture \cite{saha}.
\end{conj}
\section{Odd Euler Sums} \label{sec5}

\subsection{Definitions}

Let us begin with several definitions and notation.
Set 
\[
\mcO_n\left(p\right)=\sum_{k=1}^n\frac{1}{(2k-1)^p}\,\,\,\,\text{and}\,\,\,\, \mathcal{B}_n\left(p\right)=\sum_{k=1}^n\frac{(-1)^{k}}{(2k-1)^p}.
\]
Notice that
\[\mcO\left(p\right)=\mcO_{\infty}\left(p\right)=\left(1-2^{-p}\right)\z(p) \,\,\,\text{and}\,\,\,\mathcal{B}\left(p\right)=\mathcal{B}_{\infty}\left(p\right)=\beta\left(p\right)\]
where $\beta$ denotes the \emph{Dirichlet beta function}. In particular, 
$\beta\left(2\right)=\mathcal{G}$ (the Catalan constant). 
\begin{defn}
Call each of 
\[
\mcO\left(p, q\right)=\sum_{n=1}^\mg \frac{\mcO_n\left(p\right)}{(2n-1)^q}\,\,\,\,\text{and}\,\,\,\,\mathcal{B}\left(p, q\right)=\sum_{n=1}^\mg\frac{(-1)^{n}\,\mathcal{B}_n\left(p\right)}{(2n-1)^q}\]
an \emph{odd Euler sum}.
\end{defn}

\subsection{Reflection formula for Odd Euler Sums and Evaluation of $\mathcal{O}\left(q,q\right)$ and $\mathcal{B}\left(q,q\right)$}
Euler sums $\{H(p, q)\}$ satisfy reflection formula.  Here is a natural analog for odd Euler sums.
\begin{prop}
The following identities hold.
\begin{equation}\label{duality1}
\mcO\left(p, q\right)+\mcO\left(q, p\right)=
\mcO\left(p\right)\mcO\left(q\right)+\mcO\left(p+q\right),
\end{equation}
\begin{equation}\label{duality2}
\mathcal{B}\left(p, q\right)+\mathcal{B}\left(p,q\right)=
\beta\left(p\right)\beta\left(q\right)+\mathcal{B}\left(p + q\right).
\end{equation}
\end{prop}

\begin{proof}
Since the double sum 
\[\sum_{k}\sum_{n}\f{1}{\left(2k-1\right)^p\left(2n-1\right)^q}\]
is absolutely convergent, we can  switch the order of summation to get 
\[
	\mcO\left(p, q\right)=
\sum_{n=1}^{\mg}
\k{
\sum_{k=1}^{n}\f{1}{\left(2k-1\right)^{p}}}
\f{1}{(2n-1)^{q}}
	=
	\sum_{k=1}^{\mg}
	\sum_{n=k}^{\mg}
	\f{1}{\left(2n-1\right)^{q}}
	\f{1}{\left(2k-1\right)^{p}}\]
	
	\[=
	\sum_{k=1}^{\mg}
	\k{\sum_{n=1}^{\mg}-\sum_{n=1}^{k-1}}
\f{1}{\left(2n-1\right)^{q}}
	\f{1}{\left(2k-1\right)^{p}}
	=
	\sum_{k=1}^{\mg}\f{\left[\mcO\left(q\right)-\mcO_{k-1}\left(q\right)\right]}{\left(2k-1\right)^{p}}
	\]
	
	\[=\mcO\left(p\right)\mcO\left(q\right)-
	\sum_{k=1}^{\mg}\k{\mcO_{k}\left(q\right)-
	\f{1}{\left(2k-1\right)^{q}}
	}\f{1}{\left(2k-1\right)^{p}}
	=\mcO\left(p\right)\mcO\left(q\right)-\mcO\left(q, p\right)+\mcO\left(p, q\right).
\]
Inserting the signs $(-1)^{k-1}$ and $(-1)^{n-1}$ appropriately, we can show the other one (\ref{duality2}). 
\end{proof}

\begin{cor}
The reflection formula enables us to evaluate odd Euler sums for $p=q$.
    \begin{equation}\label{pqduality1}\mcO\left(q,q\right) = \dfrac{1}{2}\left[\left(1-2^{-2q}\right)\zeta(2q)+\left(1-2^{-q}\right)^2\zeta(q)\,\zeta(q)\right],\end{equation}
        \begin{equation}\label{pqduality2}
           \mathcal{B}\left(q,q\right) = \dfrac{1}{2}\left[\left(1-2^{-2q}\right)\zeta(2q)+\beta\left(q\right)\beta\left(q\right)\right]. \end{equation}
       \end{cor}
\begin{ex}
The following identities hold.
\[
\mcO\left(2, 2\right)=\f{1}{2}
\k{\f{15}{16}\,\z(4)+\f{9}{16}\,\z^2(2)}
=\f{5\pi^4}{384}, \,\,\,
\mathcal{B}\left(3, 3\right)=\f{1}{2}
\k{\f{31}{32}\,\z(6)+\beta^2\left(3\right)}=\dfrac{1937\pi^6}{1935360}.
\]
\end{ex}
It is thus natural to ask whether it is possible to evaluate $\mcO\left(p, q\right)$ or $\mathcal{B}\left(p, q\right)$ for $p \ne q$. We spend some time discussing this question in the next two subsections. 
\subsection{Evaluation of $\mathcal{O}\left(p,q\right)$}
We begin by evaluating $\mathcal{O}\left(2,3\right)$ by transforming $\mathcal{O}_n\left(p\right)$ to an integral and then transforming the initial double sum $\mathcal{O}\left(p,q\right)$ to a double integral followed by its reduction to single integrals involving the logarithm and polylogarithm function.

 Next, we show that our approach can be generalized to obtain an integral reduction for $\mathcal{O}\left(p,q\right)$. Let's start with  a  variant of $\mathcal{O}\left(2,2\right)$: 
\[\sum_{n=1}^{+\infty}\dfrac{1}{\left(2n+1\right)^2}\sum_{k=0}^{n-1}\dfrac{1}{\left(2k+1\right)^2}=-\sum_{n=1}^{+\infty}\dfrac{1}{\left(2n+1\right)^2}\int_0^1\left[\sum_{k=0}^{n-1}z^{2k}\right]\log\left(z\right)\mathrm{d}z\]

\[=-\sum_{n=1}^{+\infty}\dfrac{1}{\left(2n+1\right)^2}\int_0^1\dfrac{\log\left(z\right)}{1-z^2}\,\mathrm{d}z + \int_0^1\left[\sum_{n=1}^{+\infty}\dfrac{z^{2n}}{\left(2n+1\right)^2}\right]\dfrac{\log\left(z\right)}{1-z^2}\,\mathrm{d}z\]

\[=\dfrac{3}{4}\,\zeta(2)\sum_{n=1}^{+\infty}\dfrac{1}{\left(2n+1\right)^2}-\int_0^1\left[\int_0^1\sum_{n=1}^{+\infty}\left(xz\right)^{2n}\log\left(x\right)\mathrm{d}x\right]\dfrac{\log\left(z\right)}{1-z^2}\,\mathrm{d}z\]
%

\[=\dfrac{3}{4}\,\zeta(2)\sum_{n=1}^{+\infty}\dfrac{1}{\left(2n+1\right)^2}-\int_0^1 x^2 \log\left(x\right)\int_0^1\dfrac{z^2\log\left(z\right)}{\left(1-x^2z^2\right)\left(1-z^2\right)}\,\mathrm{d}z\,\mathrm{d}x\]

\[=\dfrac{3}{4}\zeta(2)\sum_{n=1}^{+\infty}\dfrac{1}{\left(2n+1\right)^2} - \int_0^1\dfrac{x\log(x)\mathrm{Li}_2(x)}{1-x^2}\mathrm{d}x+\dfrac{1}{4} \int_0^1\dfrac{x\log(x)\mathrm{Li}_2(x^2)}{1-x^2}\mathrm{d}x+\dfrac{\pi^2}{8}\int_0^1\dfrac{x^2\log(x)}{1-x^2}\mathrm{d}x\]
\[=\dfrac{3}{4}\zeta(2)\sum_{n=1}^{+\infty}\dfrac{1}{\left(2n+1\right)^2} - \int_0^1\dfrac{x\log(x)\mathrm{Li}_2(x)}{1-x^2}\mathrm{d}x+\dfrac{1}{4} \int_0^1\dfrac{x\log(x)\mathrm{Li}_2(x^2)}{1-x^2}\mathrm{d}x-\dfrac{3}{4}\zeta(2)\sum_{n=1}^{+\infty}\dfrac{1}{\left(2n+1\right)^2} \]

\[= \dfrac{1}{4} \int_0^1\dfrac{x\log(x)\mathrm{Li}_2(x^2)}{1-x^2}\mathrm{d}x-\int_0^1\dfrac{x\log(x)\mathrm{Li}_2(x)}{1-x^2}\mathrm{d}x= \dfrac{1}{2} \int_0^1\dfrac{\log(x)\mathrm{Li}_2(x)}{1+x}\mathrm{d}x-\dfrac{7}{16}\int_0^1\dfrac{\log(x)\mathrm{Li}_2(x)}{1-x}\mathrm{d}x\]
\begin{equation}=\dfrac{1}{2}\left[-\dfrac{3}{16}\,\zeta(4)\right] -\dfrac{7}{16}\left[-\dfrac{3}{4}\,\zeta(4)\right]= -\dfrac{3}{32}\,\zeta(4) + \dfrac{21}{64}\,\zeta(4) = \dfrac{15}{64}\,\zeta(4)=\f{\pi^4}{384}.\end{equation}
Similarly, we have
\[\mcO\left(2,3\right)=\sum_{n=1}^{+\infty}\dfrac{1}{\left(2n-1\right)^3}\sum_{k=1}^{n}\dfrac{1}{\left(2k-1\right)^2}=-\sum_{n=1}^{+\infty}\dfrac{1}{\left(2n-1\right)^3}\int_0^1\left[\sum_{k=1}^nz^{2k}\right]\dfrac{\log\left(z\right)}{z^2}\,\mathrm{d}z\]

\[=-\dfrac{1}{2}\int_0^1\dfrac{\log^{2}(x)}{x^2}\int_0^1\left[\sum_{n=1}^{+\infty}x^{2n}\left(1-z^{2n}\right)\right]\dfrac{\log\left(z\right)}{1-z^2}\,\mathrm{d}z\,\mathrm{d}x\]

\[=-\dfrac{1}{2}\int_0^1\dfrac{\log^{2}(x)}{1-x^2}\int_0^1\dfrac{\log\left(z\right)}{1-x^2z^2}\,\mathrm{d}z\,\mathrm{d}x=\dfrac{1}{4} \int_0^1\dfrac{\log^2(x)\,\mathrm{Li}_2\left(x\right)}{x\left(1-x^2\right)}\mathrm{d}x-\dfrac{1}{4}\int_0^1\dfrac{\log^2(x)\,\mathrm{Li}_2\left(-x\right)}{x\left(1-x^2\right)}\mathrm{d}x\]
\begin{equation}\label{gfvrify}=\dfrac{1}{4}\left(\dfrac{11}{16}\,\zeta(5) + \dfrac{3}{4}\,\zeta(2)\,\zeta(3)\right) - \dfrac{1}{4}\left(-\dfrac{5}{4}\,\zeta(5) - \dfrac{3}{8}\,\zeta(2)\,\zeta(3)\right)=\dfrac{31}{64}\,\zeta(5) + \dfrac{9}{32}\,\zeta(3)\,\zeta(2).\end{equation}
This approach can be easily generalized to the following. 
\begin{thm}\label{general_int1}
Let $p,q\in \NN$ and $\mathrm{Li}_n(z)$ represent the polylogarithm, then we have
\[\mathcal{O}\left(p,q\right)  =\dfrac{\left(-1\right)^{q}}{2\left(q-1\right)!}\left[\int_0^1\dfrac{\log^{q-1}\left(x\right)\mathrm{Li}_p\left(-x\right)}{x\left(1-x^2\right)}\,\mathrm{d}x-\int_0^1\dfrac{\log^{q-1}\left(x\right)\mathrm{Li}_p\left(x\right)}{x\left(1-x^2\right)}\,\mathrm{d}x \right].\]
\end{thm}
\begin{proof}
As earlier, we begin by transforming  $\mathcal{O}_n\left(p\right)$ to an integral
\[\mcO\left(p,q\right)=\sum_{n=1}^{+\infty}\dfrac{1}{\left(2n-1\right)^q}\sum_{k=1}^{n}\dfrac{1}{\left(2k-1\right)^p}=\sum_{n=1}^{+\infty}\dfrac{1}{\left(2n-1\right)^q}\dfrac{\left(-1\right)^{p-1}}{\left(p-1\right)!}\int_0^1\left[\sum_{k=1}^nz^{2k}\right]\dfrac{\log^{p-1}(z)}{z^2}\,\mathrm{d}z\]

\[=\dfrac{\left(-1\right)^{q-1}}{\left(q-1\right)!}\dfrac{\left(-1\right)^{p-1}}{\left(p-1\right)!}\int_0^1\dfrac{\log^{q-1}(x)}{x^2}\int_0^1\left[\sum_{n=1}^{+\infty}x^{2n}\left(1-z^{2n}\right)\right]\dfrac{\log^{p-1}(z)}{1-z^2}\,\mathrm{d}z\,\mathrm{d}x\]

\[=\dfrac{\left(-1\right)^{q-1}}{\left(q-1\right)!}\dfrac{\left(-1\right)^{p-1}}{\left(p-1\right)!}\int_0^1\dfrac{\log^{q-1}(x)}{1-x^2}\int_0^1\dfrac{\log^{p-1}(z)}{1-x^2z^2}\,\mathrm{d}z\,\mathrm{d}x\]

\[=\dfrac{\left(-1\right)^{q-1}}{2\left(q-1\right)!}\int_0^1\dfrac{\log^{q-1}(x)}{1-x^2}\left[\dfrac{\left[\mathrm{Li}_p\left(x\right)-\mathrm{Li}_p\left(-x\right)\right]}{x}\right]\mathrm{d}x\]

\[=\dfrac{\left(-1\right)^{q}}{2\left(q-1\right)!}\left[\int_0^1\dfrac{\log^{q-1}\left(x\right)\mathrm{Li}_p\left(-x\right)}{x\left(1-x^2\right)}\,\mathrm{d}x-\int_0^1\dfrac{\log^{q-1}\left(x\right)\mathrm{Li}_p\left(x\right)}{x\left(1-x^2\right)}\,\mathrm{d}x \right]\]
as desired. 
\end{proof}
\begin{rem} Note that \textsc{Piscos Mathematica Package} produces the integral evaluations
\[\int_0^1\dfrac{\log^{3}\left(z\right)\mathrm{Li}_3\left(z\right)}{z\left(1-z^2\right)}\,\mathrm{d}z= -\dfrac{3\pi^4}{64}\,\zeta(3) + \dfrac{5\pi^2}{16}\,\zeta(5) - \dfrac{489}{128}\,\zeta(7),\]
	
\[\int_0^1\dfrac{\log^{3}\left(z\right)\mathrm{Li}_3\left(-z\right)}{z\left(1-z^2\right)}\,\mathrm{d}z = \dfrac{3\pi^4}{64}\,\zeta(3) - \dfrac{5\pi^2}{32}\,\zeta(5) + \dfrac{273}{128}\,\zeta(7),\]

\[\int_0^1\dfrac{\log^{4}\left(z\right)\mathrm{Li}_4\left(z\right)}{z\left(1-z^2\right)}\,\mathrm{d}z = \dfrac{\pi^4}{24}\,\zeta(5) + \dfrac{35\pi^2}{32}\,\zeta(7) + \dfrac{579}{64}\,\zeta(9),\]

\[\int_0^1\dfrac{\log^{4}\left(z\right)\mathrm{Li}_4\left(-z\right)}{z\left(1-z^2\right)}\,\mathrm{d}z = -\dfrac{7\pi^4}{192}\,\zeta(5) - \dfrac{35\pi^2}{64}\,\zeta(7) - \dfrac{477}{32}\,\zeta(9),\]

\[\int_0^1\dfrac{\log^{5}\left(z\right)\mathrm{Li}_5\left(z\right)}{z\left(1-z^2\right)}\,\mathrm{d}z = -\dfrac{15\pi^6}{128}\,\zeta(5) + \dfrac{7\pi^4}{32}\,\zeta(7) + \dfrac{315\pi^2}{64}\,\zeta(9) - \dfrac{18825}{256}\,\zeta(11),\]

\[\int_0^1\dfrac{\log^{5}\left(z\right)\mathrm{Li}_5\left(-z\right)}{z\left(1-z^2\right)}\,\mathrm{d}z = \dfrac{15\pi^6}{128}\,\zeta(5) - \dfrac{49\pi^4}{256}\,\zeta(7) - \dfrac{315\pi^2}{128}\,\zeta(9) + \dfrac{1485}{32}\,\zeta(11),\]

\[\int_0^1\dfrac{\log^{6}\left(z\right)\mathrm{Li}_6\left(z\right)}{z\left(1-z^2\right)}\,\mathrm{d}z = \dfrac{\pi^6}{24}\,\zeta(7) + \dfrac{21\pi^4}{16}\,\zeta(9) + \dfrac{3465\pi^2}{128}\,\zeta(11) + \dfrac{72855}{256}\,\zeta(13),\]

\[\int_0^1\dfrac{\log^{6}\left(z\right)\mathrm{Li}_6\left(-z\right)}{z\left(1-z^2\right)}\,\mathrm{d}z = -\dfrac{31\pi^6}{768}\,\zeta(7) - \dfrac{147\pi^4}{128}\,\zeta(9) - \dfrac{3465\pi^2}{256}\,\zeta(11) - \dfrac{222885}{512}\,\zeta(13).\]
Combining these integral evaluations with Theorem \ref{general_int1} yields
   \[\mcO\left(3,4\right) = \dfrac{\pi^4}{128}\,\zeta(3) - \dfrac{5\pi^2}{128}\,\zeta(5) + \dfrac{127}{256}\,\zeta(7),\]

\[\mcO\left(4,5\right) = \dfrac{5\pi^4}{3072}\,\zeta(5) + \dfrac{105\pi^2}{3072}\,\zeta(7) + \dfrac{511}{1024}\,\zeta(9),\]

\[\mcO\left(5,6\right) = \dfrac{\pi^6}{1024}\,\zeta(5) - \dfrac{7\pi^4}{4096}\,\zeta(7) - \dfrac{63\pi^2}{2048}\,\zeta(9) + \dfrac{2047}{4096}\,\zeta(11),\]
\[\mcO\left(6,7\right) = \dfrac{7\pi^6}{122880}\,\zeta(7) + \dfrac{7\pi^4}{4096}\,\zeta(9)+\dfrac{231\pi^2}{8192}\,\zeta(11) + \dfrac{8191}{16384}\,\zeta(13).\]
\\
Consequently, using the reflection formula we have
\[\mcO\left(4,3\right)=\dfrac{\pi^4}{728}\,\zeta(3) + \dfrac{5\pi^2}{128}\,\zeta(5) + \dfrac{127}{256}\,\zeta(7),\]

\[\mcO\left(5,4\right) = \dfrac{13\pi^4}{1536}\,\zeta(5) - \dfrac{105\pi^2}{3072}\,\zeta(7) + \dfrac{511}{1024}\,\zeta(9),\]

\[\mcO\left(6,5\right) = \dfrac{\pi^6}{30720}\,\zeta(5) + \dfrac{7\pi^4}{4096}\,\zeta(7) + \dfrac{63\pi^2}{2048}\,\zeta(9) + \dfrac{2047}{4096}\,\zeta(11),\]

\[\mcO\left(7,6\right) = \dfrac{\pi^6}{1024}\,\zeta(7) - \dfrac{7\pi^4}{4096}\,\zeta(9)-\dfrac{231\pi^2}{8192}\,\zeta(11) + \dfrac{8191}{16384}\,\zeta(13).\]
\\
However, note that in these specific values of $\mcO\left(p,q\right)$, no simple pattern can be discerned. \end{rem}
\subsection{Evaluation of $\mathcal{B}\left(p,q\right)$}
We begin by evaluating $\mathcal{B}\left(2,3\right)$ the same way we evaluated $\mathcal{O}\left(2,3\right)$ in the previous subsection. Next, we show that our approach can be generalized to obtain an integral reduction for $\mathcal{B}\left(p,q\right)$. Let's begin with $\mathcal{B}\left(2,3\right)$: 
\[\mathcal{B}\left(2,3\right) = \sum_{n=1}^{+\infty}\dfrac{\left(-1\right)^{n-1}}{\left(2n-1\right)^3}\sum_{k=1}^{n}\dfrac{\left(-1\right)^{k-1}}{\left(2k-1\right)^2}=\sum_{n=1}^{+\infty}\dfrac{\left(-1\right)^{n-1}}{\left(2n-1\right)^3}\int_0^1\left[\sum_{k=1}^n\left(-z^2\right)^k\right]\dfrac{\log\left(z\right)}{z^2}\,\mathrm{d}z\]

\[=\sum_{n=1}^{+\infty}\dfrac{\left(-1\right)^{n-1}}{\left(2n-1\right)^3}\int_0^1\dfrac{\left(-z^2\right)^n\log\left(z\right)}{1+z^2}\,\mathrm{d}z-\sum_{n=1}^{+\infty}\dfrac{\left(-1\right)^{n-1}}{\left(2n-1\right)^3}\int_0^1\dfrac{\log\left(z\right)}{1+z^2}\,\mathrm{d}z\]

\[=-\dfrac{1}{2}\int_0^1\dfrac{\log^2\left(x\right)}{x^2}\int_0^1\left[\sum_{n=1}^{+\infty}\left(-x^2\right)^n\left(-z^2\right)^n\right]\dfrac{\log\left(z\right)}{1+z^2}\,\mathrm{d}z\,\mathrm{d}x + \dfrac{\mcG \pi^3}{32}\]

\[=-\dfrac{\mcG}{2}\int_0^1\dfrac{\log^2\left(x\right)}{1+x^2}\,\mathrm{d}x - \dfrac{1}{8}\int_0^1\dfrac{\log^2\left(x\right)\mathrm{Li}_2\left(x^2\right)}{x\left(1+x^2\right)}+\dfrac{1}{2}\int_0^1\dfrac{\log^2\left(x\right)\mathrm{Li}_2\left(x\right)}{x\left(1+x^2\right)}\,\mathrm{d}x + \dfrac{\mcG \pi^3}{32}\]

\[= \dfrac{\mcG \pi^3}{32} - \dfrac{1}{8}\int_0^1\dfrac{\log^2\left(x\right)\mathrm{Li}_2\left(x^2\right)}{x\left(1+x^2\right)}+\dfrac{1}{2}\int_0^1\dfrac{\log^2\left(x\right)\mathrm{Li}_2\left(x\right)}{x\left(1+x^2\right)}\,\mathrm{d}x + \dfrac{\mcG \pi^3}{32}=\dfrac{1}{2}\int_0^1\dfrac{\log^2\left(x\right)\mathrm{Li}_2\left(x\right)}{x\left(1+x^2\right)}\,\mathrm{d}x \]

\[- \dfrac{1}{8}\int_0^1\dfrac{\log^2\left(x\right)\mathrm{Li}_2\left(x^2\right)}{x\left(1+x^2\right)}\,\mathrm{d}x=\dfrac{1}{2}\int_0^1\dfrac{\log^2\left(x\right)\mathrm{Li}_2\left(x\right)}{x\left(1+x^2\right)}\,\mathrm{d}x - \dfrac{1}{64}\int_0^1\dfrac{\log^2\left(x\right)\mathrm{Li}_2\left(x\right)}{x\left(1+x\right)}\,\mathrm{d}x\]
Some routine manipulations produce 
\begin{equation}\int_0^1\dfrac{\log^2\left(x\right)\mathrm{Li}_2\left(x\right)}{x\left(1+x\right)}\,\mathrm{d}x =\dfrac{83}{8}\,\zeta(5) - \dfrac{9}{2}\,\zeta(2)\,\zeta(3).\end{equation}
Next, we have
\[\int_0^1\dfrac{\log^2\left(x\right)\mathrm{Li}_2\left(x\right)}{x\left(1+x^2\right)}\,\mathrm{d}x = -\int_0^1 \int_0^1 \dfrac{\log\left(z\right)\log^2\left(x\right)}{\left(1-xz\right)\left(1+x^2\right)}\,\mathrm{d}x\,\mathrm{d}z\]

\[=-2\int_0^1\dfrac{z\log\left(z\right)\mathrm{Li}_2\left(z\right)}{1+z^2}\,\mathrm{d}z-\dfrac{3}{16}\,\zeta(3)\int_0^1\dfrac{z\log\left(z\right)}{1+z^2}\,\mathrm{d}z - \dfrac{\pi^3}{16}\int_0^1\dfrac{\log\left(z\right)}{1+z^2}\,\mathrm{d}z\]

\[=\int_0^1 \dfrac{\log\left(1+z^2\right)\mathrm{Li}_3\left(z\right)}{z}\,\mathrm{d}z-\int_0^1 \dfrac{\log\left(z\right)\log\left(1+z^2\right)\mathrm{Li}_2\left(z\right)}{z}\,\mathrm{d}z + \dfrac{3}{128}\,\zeta(3)\,\zeta(2) + \dfrac{\mcG \pi^3}{16}\]

\[=\dfrac{1}{8}\sum_{n=1}^{+\infty}\dfrac{\left(-1\right)^nH_{2n}}{n^4} + \dfrac{1}{4}\sum_{n=1}^{+\infty}\dfrac{\left(-1\right)^nH^{\left(2\right)}_{2n}}{n^3} + \dfrac{1}{4}\,\zeta(2)\sum_{n=1}^{+\infty}\dfrac{\left(-1\right)^{n+1}}{n^3}+ \dfrac{1}{2}\,\zeta(3)\sum_{n=1}^{+\infty}\dfrac{\left(-1\right)^{n+1}}{n^2}+ \dfrac{\mcG \pi^3}{16}\]
\begin{equation}+ \dfrac{3}{128}\,\zeta(3)\,\zeta(2) = \dfrac{\mcG \pi^3}{16} - \dfrac{9}{16}\,\zeta(2)\,\zeta(3)+ \dfrac{331}{256}\,\zeta(5) =\dfrac{\mcG \pi^3}{16} - \dfrac{3\pi^2}{32}\,\zeta(3) + \dfrac{331}{256}\,\zeta(5)\end{equation}
where the arising non-linear alternating harmonic series are evaluated by C. V\u{a}lean in \cite{MSE}: 
\[\sum _{n=1}^{+\infty}\frac{\left(-1\right)^{n-1} H_{2 n}^{(2)}}{n^3}=\frac{61 \pi^2}{192}\, \zeta (3)+\frac{1973 }{128}\zeta (5)+\frac{\pi^5}{16}-\frac{\pi}{128} \, \psi ^{(3)}\left(\frac{1}{4}\right)\]
\[\sum_{n=1}^{+\infty}\dfrac{\left(-1\right)^{n-1}H_{2n}}{n^4}=-\frac{\pi ^2}{3}\,\zeta (3)-\frac{437}{64}\, \zeta (5)-\frac{\pi ^5}{24}+\frac{\pi}{192}  \,\psi ^{(3)}\left(\frac{1}{4}\right)\]
 Putting all things together produces
\[\mathcal{B}\left(2,3\right) =\dfrac{331}{512}\,\zeta(5)-\dfrac{83}{512}\,\zeta(5)+ \dfrac{3\pi^2}{256}\,\zeta(3)- \dfrac{3\pi^2}{64}\,\zeta(3)+ \dfrac{\mathcal{G}\pi^3}{32} =\dfrac{31}{64}\,\zeta(5) - \dfrac{9\pi^2}{256}\,\zeta(3) + \dfrac{\mathcal{G}\pi^3}{32}\]
where $\mcG$ is the Catalan constant. This approach can be easily generalized to the following.
\begin{thm}\label{general_int2}
Let $p,q\in \NN$ and $\mathrm{Li}_n(z)$ represent the polylogarithm, then we have
\[\mathcal{B}\left(p,q\right)  =\dfrac{\left(-1\right)^{q}}{2\left(q-1\right)!}\left[\int_0^1\dfrac{\log^{q-1}\left(x\right)\mathrm{Li}_p\left(-x\right)}{x\left(1+x^2\right)}\,\mathrm{d}x-\int_0^1\dfrac{\log^{q-1}\left(x\right)\mathrm{Li}_p\left(x\right)}{x\left(1+x^2\right)}\,\mathrm{d}x \right].\]
\end{thm}
\begin{proof}
The proof is quite similar to that of Theorem \ref{general_int1}, we have
\[\mathcal{B}\left(p,q\right)=\sum_{n=1}^{+\infty}\dfrac{\left(-1\right)^n}{\left(2n-1\right)^q}\sum_{k=1}^{n}\dfrac{\left(-1\right)^k}{\left(2k-1\right)^p}=\sum_{n=1}^{+\infty}\dfrac{\left(-1\right)^n}{\left(2n-1\right)^q}\dfrac{\left(-1\right)^{p-1}}{\left(p-1\right)!}\int_0^1\left[\sum_{k=1}^{n}\left(-z^2\right)^k\right]\dfrac{\log^{p-1}\left(z\right)}{z^2}\,\mathrm{d}z\]

\[=\dfrac{\left(-1\right)^q}{\left(q-1\right)!}\dfrac{\left(-1\right)^p}{\left(p-1\right)!}\int_0^1\dfrac{\log^{q-1}\left(x\right)}{x^2}\int_0^1\left[\sum_{n=1}^{+\infty}\left(-x^2\right)^n\left(\left(-z^2\right)^n-1\right)\right]\dfrac{\log^{p-1}\left(z\right)}{1+z^2}\,\mathrm{d}z\,\mathrm{d}x\]

\[=\dfrac{\left(-1\right)^q}{\left(q-1\right)!}\dfrac{\left(-1\right)^p}{\left(p-1\right)!}\int_0^1\dfrac{\log^{q-1}\left(x\right)}{1+x^2}\int_0^1\dfrac{\log^{p-1}\left(z\right)}{1-x^2z^2}\,\mathrm{d}z\,\mathrm{d}x\]
\[=\dfrac{\left(-1\right)^{q-1}}{2\left(q-1\right)!}\int_0^1\dfrac{\log^{q-1}(x)}{1+x^2}\left[\dfrac{\left[\mathrm{Li}_p\left(x\right)-\mathrm{Li}_p\left(-x\right)\right]}{x}\right]\mathrm{d}x\]

\[=\dfrac{\left(-1\right)^{q}}{2\left(q-1\right)!}\left[\int_0^1\dfrac{\log^{q-1}\left(x\right)\mathrm{Li}_p\left(-x\right)}{x\left(1+x^2\right)}\,\mathrm{d}x-\int_0^1\dfrac{\log^{q-1}\left(x\right)\mathrm{Li}_p\left(x\right)}{x\left(1+x^2\right)}\,\mathrm{d}x \right]\]
as desired. 
\end{proof}

\section{Further Research and Acknowledgements}
Several paths have not been explored yet and will be the subject of future work. Here, we record some ideas for our future research.
\\
(1) We can regard our inverse sine integral as a \emph{log-sine  integral}:
\[
I(n)=
\int_0^1 \ff{\arcsin^n(z)}{z}\,\mathrm{d}z
=-
n
\int_0^1  \ff{\log (z)\arcsin^{n-1}(z)}{\sqrt{1-z^2}}\,\mathrm{d}z
=
-n
\int_0^{\pi/2}  {z^{n-1}\log(\sin (z))}\,\mathrm{d}z.\]
We refer the reader to Borwein--Broadhurst--Kamnitzer \cite{bobrka} and  Williams-Yue \cite{wiyu} for relations of such log-sine integrals with central binomial series. Notice that they discussed 
\[
\int_0^{\pi/3} \log^\alpha \k{2\sin \left(\f{\vartheta}{2}\right)}\vartheta^\beta \,\mathrm{d}\vartheta
\quad \left(\alpha, \beta\geqslant 0\right)
\]
which is not quite same to ours but it should be possible to relate $I(n)$ with other central binomial series. This may serve as a subject of future work.
\\
(2) Together with multiple zeta values and  Euler sums evaluated in this article, we can compute  other MZVs. For example, we may show that
\begin{equation}\zeta(3,1,1) = 2\,\zeta(5) - \zeta(2)\,\zeta(3).\end{equation}
Indeed, 
V\u{a}lean \cite[page 303, problem 4.44]{valean} shows that
\[\sum_{n=1}^{+\infty}\dfrac{H_n}{n}\left(\zeta(3)-1-\dfrac{1}{2^3} - \dfrac{1}{3^3} - \cdots - \dfrac{1}{n^3}\right) = 2\,\zeta(2)\,\zeta(3) - \dfrac{7}{2}\,\zeta(5)\]
where $H_n$ is the $n$-th Harmonic number. Notice that the left--hand side is equal to
\[\sum_{m>n\geqslant k}\dfrac{1}{m^3nk} = \sum_{m>n> k}\dfrac{1}{m^3nk}+\sum_{m>n=k}\dfrac{1}{m^3nk}=\zeta(3,1,1) + \zeta(3,2)\]
Hence, with $\zeta(3,2)$ we just found above (\ref{spec}), we deduce that
\[\zeta(3,1,1) = 2\,\zeta(2)\,\zeta(3) -3\,\zeta(2)\,\zeta(3) - \dfrac{7}{2}\,\zeta(5) + \dfrac{11}{2}\,\zeta(5)=2\thinspace\zeta(5) - \zeta(2)\,\zeta(3).\]
(3) For a multi-index 
\[
\mathbf{i}=(a_1+1, \underbrace{1, \dots, 1}_{b_1-1}, a_2+1, \underbrace{1, \dots, 1}_{b_2-1}, \dots, a_k+1, \underbrace{1, \dots, 1}_{b_k-1}),
\]
with $a_i, b_i\geqslant1$, we define its \emph{dual}  to be
\[
\mathbf{i}^\dagger=(b_k+1, \underbrace{1, \ldots, 1}_{a_k-1}, b_{k-1}+1, \underbrace{1, \dots, 1}_{a_{k-1}-1}, \dots, b_1+1, \underbrace{1, \dots, 1}_{a_1-1}).
\]
Duality formula for MZVs claims that 
$\z(\mathbf{i})=\z(\mathbf{i}^\dagger)$ for all such indices. 

In particular, $\z(2, \{1\}^{n-1})=\z(n+1)$. 
After the preparation of this manuscript, we found out that 
Kaneko and Tsumura \cite{kats} introduced 
a \emph{multiple $T$-value}
\[
T\left(i_k, i_{k-1}, \ldots,i_1\right)
=
2^k
\sum_{
	\substack{n_k>n_{k-1}>\cdots >n_1\\
		n_j \,\,\equiv \,\,j \,\,\textrm{(mod $2$)}
}}
\dfrac{1}{n_k^{i_k}n_{k-1}^{i_{k-1}} \ldots n_1^{i_1}}
\]
with the factor $2^k$ for normalization (we changed the sum  convention to ours). They further showed that exactly the same duality formula hold for multiple $T$-values.
In particular, $T(2, \{1\}^{n-1})=T(n+1)$. 
This is equivalent to Theorem \ref{e211}, thus, we gave another proof of a special case of $T$-duality  by integration of powers of inverse hyperbolic tangent function.
\\
(4) These duality formulas come from  \emph{iterated integral expressions} for multiple zeta values with two kinds of integrals:
\[
\int \f{\mathrm{d}z}{1-z} \,\,\,\text{and}\,\,\, 
\int\f{\mathrm{d}z}{z}.
\]
For example, the famous Euler-Goldbach relation $\z(2, 1)=\z(3)$ is nothing but 
\[
\int_0^{1}\f{\mathrm{d}x_3}{x_3}
\int_0^{x_3}\f{\mathrm{d}x_2}{1-x_2}
\int_0^{x_2}\f{\mathrm{d}x_1}{1-x_1}
=
\int_0^{1}\f{\mathrm{d}z_3}{z_3}
\int_0^{z_3}\f{\mathrm{d}z_2}{z_2}
\int_0^{z_2}\f{\mathrm{d}z_1}{1-z_1}.
\]
Indeed, it should be possible to bring this idea of iterated integrals into the
study of multiple $t$-values, multiple zeta values and central binomial series with  more integrals in the form 
\[
\int \ff{\mathrm{d}z}{1-z^2}, 
\int \ff{\mathrm{d}z}{z\left(1-z^2\right)}\,\,\,\text{and}\,\,\, 
\int \ff{\mathrm{d}z}{\sqrt{1-z^2}}.\]
Definition \ref{W}, Proposition \ref{l1} and Proposition \ref{even} implicitly highlight this little idea.  

Notice that there exist iterated integral expressions for integer powers of inverse sine function and inverse hyperbolic tangent function:
\[
\f{\arcsin^n (z)}{n!}=
\int_0^z \f{\mathrm{d}z_1}{\sqrt{1-z_1^2}}
\int_0^{z_1} \f{\mathrm{d}z_2}{\sqrt{1-z_2^2}}
\cdots
\int_0^{z_{n-1}} \f{\mathrm{d}z_n}{\sqrt{1-z_n^2}}
\]
as remarked in \cite[(4.1)]{boch}. Quite similarly, we can show that 
\[
\f{\text{arctanh}^n (z)}{n!}=
\int_0^z \f{\mathrm{d}z_1}{{1-z_1^2}}
\int_0^{z_1} \f{\mathrm{d}z_2}{{1-z_2^2}}
\cdots
\int_0^{z_{n-1}} \f{\mathrm{d}z_n}{{1-z_n^2}}.
\]
Thus, as observed above, the equality 
\[
t\left(2, 1\right)=\f{I(2)}{2!}=
\dfrac{1}{2}\int_0^1 \f{\arcsin^2(z)}{z}\,\mathrm{d}z
\]
implies the relation of iterated integrals
\[
\int_0^{1}\f{\mathrm{d}x_3}{x_3}
\int_0^{x_3}\f{\mathrm{d}x_2}{x_2\left(1-x_2^2\right)}
\int_0^{x_2}\f{\mathrm{d}x_1}{1-x_1^2}
=
\int_0^{1}\f{\mathrm{d}z_3}{z_3}
\int_0^{z_3}\f{\mathrm{d}z_2}{\sqrt{1-z_2^2}}
\int_0^{z_2}\f{\mathrm{d}z_1}{\sqrt{1-z_1^2}}
\]
which is quite remarkable and not so obvious at a first glance. This can be indeed extended and we wish to study more about such iterated integral relations at some another opportunity. 
\\
(5) It remains an open problem to compute a closed expression for the integrals
\[\int_0^1\dfrac{\log^{q-1}\left(z\right)\mathrm{Li}_p\left(z\right)}{z\left(1+z^2\right)} \,\,\,\text{and}\,\,\, \int_0^1\dfrac{\log^{q-1}\left(z\right)\mathrm{Li}_p\left(z\right)}{z\left(1-z^2\right)}\]
as it would allow us to compute $\mcO\left(p,q\right)$ (Theorem \ref{general_int1}). 


The first author would like to thank Christophe Vignat for his guidance and support throughout the completion of this work.
The second author would like to thank Satomi Abe, Yuko Takada  and Michihito Tobe for sincerely supporting his research.

\end{document}